\documentclass[11pt]{article}
\usepackage{anysize}\marginsize{3.5cm}{3.5cm}{1.5cm}{2.5cm}

\makeatletter\@ifundefined{pdfpagewidth}{}{\pdfpagewidth=21.0cm\pdfpageheight=29.7cm}\makeatother 

\usepackage{amssymb,amsthm,amsmath}
\makeatletter
\let\orig@item=\@item \def\@item[#1]{\orig@item[\rm #1]}
\makeatother

\theoremstyle{plain}
\newtheorem{theorem}{Theorem}[section]
\newtheorem{lemma}[theorem]{Lemma}
\newtheorem{conjecture}[theorem]{Conjecture}
\newtheorem{corollary}[theorem]{Corollary}

\theoremstyle{definition}
\newtheorem{definition}[theorem]{Definition}
\newtheorem{example}[theorem]{Example}

\newtheorem{remark}[theorem]{Remark}

\newcommand\be{\begin{eqnarray*}}
\newcommand\ee{\end{eqnarray*}}
\renewcommand\ge{\geqslant}
\renewcommand\geq{\geqslant}
\renewcommand\le{\leqslant}
\renewcommand\leq{\leqslant}

\let\hat=\widehat
\newcommand\C{\mathbb C}
\renewcommand\P{\mathbb P}
\newcommand\cali{{I}}  
\newcommand\calj{{J}}  
\newcommand\newop[2]{\def#1{\mathop{\rm #2}\nolimits}}
\newop\Ass{Ass}
\newop\mult{mult}
\newcommand\eqnref[1]{(\ref{#1})}
\newcommand\rounddown[1]{\left\lfloor#1\right\rfloor}
\newcommand\m{\mathfrak{m}}

\begin{document}

  \title{The effect of points fattening in dimension three}
  \author{Th.~Bauer$^1$ and T.~Szemberg$^2$}
  \date{January 12, 2014}
  \footnotetext[1]{Partially supported by DFG grant BA 1559/6-1}
  \footnotetext[2]{Partially supported by NCN grant UMO-2011/01/B/ST1/04875}
  \maketitle


\begin{abstract}
   There has been increased recent interest in understanding the
   relationship between the symbolic powers of an ideal and the
   geometric properties of the corresponding variety. While a
   number of results are available for the two-dimensional case,
   the higher-dimensional case is largely unexplored. In the
   present paper we study a natural conjecture arising from a
   result by Bocci and Chiantini. As a first step towards
   understanding the higher-dimensional picture, we show that
   this conjecture is true in dimension three. Also, we provide
   examples showing that the hypotheses of the conjecture may not
   be weakened.
\end{abstract}


\section*{Dedication}

   This article is dedicated to Robert Lazarsfeld on the
   occasion of his 60th
   birthday.


\section{Introduction}\label{sec:intro}

   The study of the effect of points fattening
   has been initiated by Bocci and Chiantini
   in \cite{BocCha11}. Roughly speaking, they consider the radical ideal $\cali$
   of a finite set $Z$ of points in the projective plane, its second
   symbolic power $\cali^{(2)}$, and deduce from the comparison of algebraic
   invariants of these two ideals various geometric properties of the set $Z$.
   Along these lines Dumnicki, Tutaj-Gasi\'nska and the second author studied
   in \cite{DST13} higher symbolic powers of~$\cali$. Similar problems
   were studied in \cite{Lanckorona} in the bi-homogeneous setting of ideals
   defining finite sets of points in $\P^1\times\P^1$.

   It is a natural task to try to generalize the result of
   Bocci and Chiantini \cite[Theorem 1.1]{BocCha11}
   to the higher-dimensional setting. Denoting for a homogeneous ideal
   $\cali$ by $\alpha(\cali)$ its \emph{initial degree}, i.e., the least
   degree $k$ such that $(\cali)_k\neq 0$,
   a natural generalization reads as follows:

\begin{conjecture}\label{conj:main}
   Let $Z$ be a finite set of points in projective space $\P^n$ and let
   $\cali$ be the radical ideal defining $Z$. If
   \begin{equation}\label{eq:jumps by 1}
      d:=\alpha(\cali^{(n)})=\alpha(\cali)+n-1
      \,,
   \end{equation}
   then either
   \begin{itemize}
      \item[] $\alpha(\cali)=1$, i.e., $Z$ is contained in a single hyperplane $H$ in $\P^n$
   \end{itemize}
   or
   \begin{itemize}
      \item[] $Z$ consists of all intersection points (i.e., points where $n$ hyperplanes meet)
      of a general configuration of $d$ hyperplanes in $\P^n$, i.e., $Z$ is a \emph{star configuration}.
      For any polynomial in $\cali^{(n)}$ of degree $d$, the corresponding hypersurface
      decomposes into $d$ such hyperplanes.
   \end{itemize}
\end{conjecture}
   The term \emph{general configuration} in the conjecture means simply that no more
   than $n$ hyperplanes meet in one point. This is equivalent to
   \emph{general linear position} for points in the dual projective
   space corresponding to
   the hyperplanes in the configuration.
   The result of Bocci and Chiantini is the case $n=2$ of this conjecture.
   As a first step towards understanding
   the higher-dimensional picture,
   we show in the present paper:

\begin{theorem}\label{thm:main}
   The conjecture is true for $n=3$.
\end{theorem}

   The assumption on the ideal $\cali$ in the theorem amounts to the
   two equalities
   \be
      \alpha(\cali^{(2)})&=&\alpha(\cali)+1 \\
      \alpha(\cali^{(3)})&=&\alpha(\cali^{(2)})+1
   \ee
   and one
   might be tempted to relax the assumptions
   to only one of them.
   In Sect.~\ref{sect:further}
   we provide examples showing, however, that neither is
   sufficient by itself to get the conclusion of the theorem.

   Star configurations are interesting objects of study in their own right.
   They are defined in \cite{GHM} as unions of linear subspaces of fixed
   codimension $c$ in projective space $\P^n$ that result as subspaces
   where exactly $c$ of a fixed finite set of general hyperplanes in $\P^n$
   intersect. The case described in Conjecture \ref{conj:main} corresponds
   thus to the $c=n$ situation. It is natural to wonder if the following
   further generalization of Conjecture \ref{conj:main} might be true:
   If $Z$ is a finite collection of linear subspaces of codimension
   $c\le n$
   in $\P^n$ with the radical ideal $\cali$ and such that
   $$d=\alpha(\cali^{(c)})=\alpha(\cali)+c-1\,,$$
   then $Z$ is either contained in a hyperplane or forms a star
   configuration of codimension $c$ subspaces. The
   recent preprint~\cite{Jan13} of Janssen deals with lines in $\P^3$
   and shows that such a simple generalization would be too naive.
   Nevertheless we expect that there are some undiscovered patterns
   lurking behind, and we hope to come back to this subject in the near future.

   Throughout the paper we work over the complex numbers, and
   we use standard notation in algebraic geometry as in
   \cite{PAG}.


\section{Initial degrees of symbolic powers}

\begin{definition}[Symbolic power]\label{def:symb power}
   Let $\cali$ be a homogeneous ideal in the polynomial ring $R=\C[\P^n]$.
   For a positive integer $k$, the ideal
   $$\cali^{(k)} = R \cap \left( \bigcap_{\mathfrak{p} \in \Ass(I)} \cali^{k} R_{\mathfrak{p}} \right),$$
   where the intersection is taken in the field of fractions of $R$,
   is the $k$-th \emph{symbolic power} of $\cali$.
\end{definition}

\begin{definition}[Differential power]\label{def:diff power}
   Let $\cali$ be a radical homogeneous ideal and let $V\subset\P^n$
   be the corresponding subvariety.
   For a positive integer $k$, the ideal
   $$\cali^{\langle k\rangle}=\bigcap\limits_{P\in V}\m_P^k\,,$$
   where $\m_P$ denotes the maximal ideal defining the point $P\in\P^n$,
   is the $k$-th \emph{differential power} of $\cali$.
\end{definition}

   In other words, the $k$-th differential power of an ideal consists of all homogeneous
   polynomials vanishing to order at least $k$ along the underlying variety.
   For radical ideals these two concepts fall together due to a result
   of Nagata and Zariski, see \cite[Theorem 3.14]{Eis95} for prime ideals
   and \cite[Corollary 2.9]{SidSul09} for radical ideals:

\begin{theorem}[Nagata, Zariski]\label{thm:NZ}
   If $\cali$ is a radical ideal in a polynomial ring over an algebraically
   closed field, then
   $$\cali^{\langle k\rangle}=\cali^{(k)}$$
   for all $k\geq 1$.
\end{theorem}

   We will make use of the following observation on symbolic powers.

\begin{lemma}\label{lem:jump must be 1 at least}
   Let $\cali$ be an arbitrary radical homogeneous ideal. Then
   we have the inequality
   $$\alpha(\cali^{(k+1)}) - \alpha(\cali^{(k)})\geq 1$$
   for all $k\geq 1$.
\end{lemma}

\begin{proof}
   Let $Z$ be the subscheme of $\P^n$ defined by $\cali$.
   The claim of the lemma follows immediately from the interpretation of symbolic
   powers as differential powers, Theorem \ref{thm:NZ},
   and the observation that if a polynomial $f$ vanishes along $Z$
   to order at least $k+1$, then any of its partial derivatives vanishes
   along $Z$ to order at least $k$.
\end{proof}


\section{The $\P^2$ case revisited}

   As a warm-up,
   we give here a new proof of the result of Bocci and Chiantini.
   This proof has the advantage that it does not make use of the
   Pl\"ucker formulas.

\begin{theorem}[Bocci-Chiantini]
   Let $Z$ be a finite set of points in the projective plane $\P^2$
   and let $\cali$ be its radical ideal. If
   $$d=\alpha(\cali^{(2)})=\alpha(\cali)+1\,,$$
   then
   either $Z$ consists of collinear points
   or $Z$ is the set of all intersection points
   of a general configuration of $d$ lines in $\P^2$.
\end{theorem}

\begin{proof}
   If $d=2$, then we are done. So we assume $d\geq 3$.

   By Lemma~\ref{lem:reduction} below
   we may assume that $Z$
   consists of exactly $\binom{d}{2}$ points.
   Let $X_2$ be a divisor of degree $d$ that is singular in all points of $Z$.
   Let $P$ be one of the points in $Z$.
   Then there exists a divisor $W_P$ of degree $d-2$ vanishing at all
   points in $Z\setminus\{P\}$ (and not vanishing at $P$).

   We claim that $W_P$ is contained in $X_2$.
   To see this, we begin by
   showing that they must
   have a common component. Indeed, this follows from B\'ezout's Theorem, since
   otherwise we would get
   $$d(d-2)=X_2\cdot W_P\geq 2\left(\binom{d}{2}-1\right)=d(d-1)-2\,,$$
   which is equivalent to $d\leq 2$
   and contradicts
   our initial assumption
   in this proof.

   Let now $\Gamma$ be the greatest common divisor of $X_2$ and $W_P$,
   and let $e$ be the degree of the divisor
   $W_P'=W_P-\Gamma$
   (so that $\deg(\Gamma)=d-2-e$).
   There must be at least $\binom{e+2}{2}-1$
   points from $Z\setminus\{P\}$ on $W_P'$ (otherwise there would
   be a pencil of such divisors
   $W_P'$ and one could choose
   an element in this pencil passing
   through $P$, but then $W_P'+\Gamma$ would be an element of degree $d-2$ in $\cali$
   contradicting the assumption on $\alpha(\cali)$).

   Then we intersect again $X_2-\Gamma$ with $W_P'$ and obtain
   $$(e+2)e=(X_2-\Gamma)\cdot W_P'\geq 2\left(\binom{e+2}{2}-1\right)=(e+2)(e+1)-2\,,$$
   which gives $e=0$.

   It follows that $X_2-W_P$ is a divisor of degree $2$ with a double point at $P$.
   Hence $X_2$ contains two lines intersecting in $P$. This holds for every point $P\in Z$.
   Since $X_2$ can contain at most $d$ lines, we see that this is only possible
   if $Z$ consists of the intersection points of a general configuration of $d$ lines.
\end{proof}


\section{A reduction result}
   We begin by a lemma concerning star configurations of points.
   We include the proof, since we were not able to trace it down
   in the literature.

\begin{lemma}\label{lem:the only one}
   Let $Z$ be a star configuration of points defined by hyperplanes
   $H_1,\ldots,H_d$ in $\P^n$. For $d\geq n+1$ the union
   $$H_1\cup\ldots\cup H_d$$
   is the only hypersurface $F$ of degree $d$ with the property
   \begin{equation}\label{eq:mult cond}
      \mult_PF\geq n\;\mbox{ for all }\; P\in Z.
   \end{equation}
\end{lemma}

\begin{proof}
   We proceed by induction on the dimension $n\geq 2$.
   The initial case of $\P^2$ is simply dealt with by
   a B\'ezout type argument. Indeed, assuming that there
   exists a curve $F$ of degree $d$ passing through all points in $Z$
   with multiplicity $\geq 2$ and taking a configuration
   line $H_i$, B\'ezout's theorem implies that $H_i$ is
   a component of $F$. Since this holds for all lines
   in the configuration and $\deg(F)=d$, we are done.

   For the induction step we assume that the lemma
   holds for dimension $n-1$ and all $d\geq n$. We want
   to conclude that it holds for $\P^n$ and all $d\geq n+1$.
   Of course we may assume that $n\geq 3$.

   To this end let $F$ be a hypersurface of degree $d$ in $\P^n$
   satisfying \eqnref{eq:mult cond}. Suppose that there
   exists a hyperplane $H$ among $H_1,\ldots,H_d$, which is not
   a component of $F$. Then the restriction $G=F\cap H$ is a hypersurface
   of degree $d$ in $H\simeq\P^{n-1}$ with $\mult_PG\geq n$
   for all $P\in Z_H=Z\cap H$. Note that $Z_H$ is itself
   a star configuration of points in $H$, defined by hyperplanes
   obtained as intersections $H_i\cap H$. So it is a star
   configuration of $d-1$ hyperplanes in $\P^{n-1}$.
   The polar system of $G$ (i.e., the linear
   system defined by all first order derivatives of the
   equation of $G$) is of dimension $n-1\geq 2$,
   and every element $K$ in this system
   satisfies
   $$\mult_PK\geq n-1\quad\mbox{ for all } P\in Z_H.$$
   This contradicts the induction assumption.
\end{proof}

   The following lemma allows one to assume $\#Z=\binom dn$ when
   proving the theorem (or when working on the conjecture).

\begin{lemma}\label{lem:reduction}
   Suppose that the set $Z\subset\P^n$ satisfies the assumptions
   of Conjecture~\ref{conj:main}
   and that $\alpha(I)\ge 2$.
   Then there is a subset $W\subset
   Z$ with the following properties:
   \begin{itemize}
   \item[(i)]
      $W$ is of cardinality $\binom dn$.
   \item[(ii)]
      For the
      ideal $\calj$ of $W$ we have
      $\alpha(\calj^{(k)})=\alpha(\cali^{(k)})$
      for $k=1,\dots,n$.
   \item[(iii)]
      If $W$ is a star configuration, then $W=Z$.
   \end{itemize}
\end{lemma}

\begin{proof}
   To begin with, note that the equality in \eqnref{eq:jumps by 1}
   together with Lemma~\ref{lem:jump must be 1 at least}
   implies $\alpha(\cali^{(k)})=d-n+k$ for $k=1,\dots,n$.

   The assumption $\alpha(\cali)\geq 2$
   implies $d\geq n+1$. Since there is no form of degree $\leq d-n$
   vanishing along $Z$, there must be at least
   \begin{equation}\label{eq:lower bound on s}
      s:= \binom{d}{n}
   \end{equation}
   points in $Z$.

   We choose now exactly $s$ points $P_1,\dots,P_s$ from $Z$ that
   impose independent
   conditions on forms of degree $d-n$. (This can be done, since
   vanishing at each point in $Z$ gives a linear equation on the coefficients
   of a form of degree $d-n$, so that we obtain a system of $\#Z$ linear
   equations of rank $s=\binom{d}{n}$ (which is the maximal
   possible rank). We choose then a subsystem
   of $s$ equations with maximal rank.) Let $W:=\left\{P_1,\dots,P_s\right\}$
   and let $\calj$ be the radical ideal of $W$.
   Since $W\subset Z$, we certainly have
   $$\alpha(\calj^{(k)})\leq \alpha(\cali^{(k)})$$
   for all $k\geq 0$.
   On the other hand we have
   $$\alpha(\calj)=d-n+1=\alpha(\cali)$$
   by the selection of $W$.
   Lemma \ref{lem:jump must be 1 at least} implies then that in fact
   $$\alpha(\calj^{(k)})=\alpha(\cali^{(k)})$$
   for $k=1,\dots,n$.
   This shows that conditions (i) and (ii) are satisfied.

   As for (iii):
   Suppose that
   $W$ is a star configuration.
   By (ii) we have
   $\alpha(I^{(n)})=\alpha(J^{(n)})$,
   hence
   it follows
   from Lemma~\ref{lem:the only one}
   that $W=Z$.
\end{proof}

   Further, we will need the following
   elementary lemma on hypersurfaces that are obtained by taking derivatives.

\begin{lemma}\label{lem:one mult remains}
  Let $X\subset\P^n$ be a hypersurface defined by a polynomial $f$ of degree $d$ with a point $P$ of
  multiplicity $m$ such that $2\leq m<d$. Then there exists a direction $v$ 
  such that the hypersurface $\Lambda$ defined by the directional derivative of $f$
  in direction $v$ has multiplicity~$m$ at $P$.
\end{lemma}

\begin{proof}
   After a projective change of coordinates
   we may assume $P=(1:0:\ldots:0)$. Then we can write
   \be
      f(x_0:x_1:\ldots:x_n) & = & x_0^{d-m}g_m(x_1:\ldots:x_n) \\
                            &  & {}+x_0^{d-m-1}g_{m+1}(x_1:\ldots:x_n) \\
                            &  & {}+\ldots \\
                            &  & {}+g_d(x_1:\ldots:x_n)
   \ee
   with homogeneous polynomials $g_i$ of degree $i$ for $i=m,\ldots,d$.
   Since $d>m$, the divisor defined by $\frac{\partial f}{\partial x_0}=0$
   has multiplicity $m$ at $P$.
\end{proof}


\section{Dimension $3$}

   In this section we give the

\begin{proof}[Proof of Theorem \ref{thm:main}]
   We proceed by induction on $d$.
   For $d\le 3$ the statement of the theorem is trivially
   satisfied, so we assume
   $d\ge 4$ now.
   By Lemma~\ref{lem:reduction} we may assume
   the $Z$ is of cardinality $\binom d3$.
   Let $X_3\subset\P^3$ be the divisor defined by a polynomial of degree $d$ in $\cali^{(3)}$.
   We assert that
   $$
      \mbox{$X_3$ is reducible.}
      \eqno(*)
   $$
   To see this, we first note that thanks to $m=3<4\le d$
   there is by Lemma~\ref{lem:one mult remains}
   for any $P\in Z$ a directional derivative surface $\Lambda_P$ of degree $d-1$
   with multiplicity at least $3$ at $P$.
   Arguing by contradiction, we assume that $X_3$ is irreducible, which
   implies that $X_3$ and $\Lambda_P$ intersect properly, i.e., in a curve.
   Adapting the proof of \cite[Proposition 3.1]{HaHu} to dimension $3$,
   we see that the linear system of forms of degree $d-2$ vanishing along $Z$
   has only $Z$ as its base locus.
   (This is due to the fact that the regularity of $\cali$ is $d-2$.)
   We can therefore choose an element $Y$
   in this system that does not contain
   any component of
   the intersection curve of $X_3$ and $\Lambda_P$.
   Then the three surfaces
   $X_3$, $\Lambda_P$ and $Y$ intersect in points only, and we can apply
   B\'ezout's theorem to get
   $$
      d(d-1)(d-2)=X\cdot\Lambda_P\cdot Y\geq 6\left(\binom{d}{3}-1\right)+9
      \,.
   $$
   But this implies $0\geq 3$, a contradiction. So $(*)$ is established.

   Let now $\Gamma$ be an irreducible component of $X_3$ of smallest
   degree. Set $\gamma=\deg\Gamma$ and $X_3'=X_3-\Gamma$.
   Our aim is to apply induction on $X_3'$.
   To this end we consider the set
   $$
      Z'=Z\setminus\Gamma
      \,.
   $$
   It is non-empty, as otherwise $Z$ would be contained in $\Gamma$
   and then $\alpha(\cali)$
   would be less than $d-2$:
   Indeed, we would have $\alpha(I)\le\gamma\le\rounddown{\frac d2}$,
   which is less than $d-2$ if $d\ge 5$; and if $d=4$, then $\gamma=2$,
   so $X_3$ consists of two quadric, which implies that
   it can have only two triple points -- but then $\alpha(I)=1$.

   As $Z'$ is non-empty, there is in particular
   a triple point on $X_3'$, and hence
   $d-\gamma=\deg X_3'\ge 3$.

   We claim that
   \begin{equation}\label{eqn:alpha-Z'}
      \alpha(\cali_{Z'})\ge d-\gamma-2
      \,.
   \end{equation}
   In fact, there is otherwise a surface $S$ of degree $d-\gamma-3$
   passing through $Z'$, and then
   $S+\Gamma$ is a divisor of degree $d-3$ passing through $Z$,
   which contradicts the assumption $\alpha(\cali)=d-2$.

   Next, note that
   \begin{equation}\label{eqn:alpha-Z'-3}
      \alpha(\cali_{Z'}^{(3)})\le\deg X_3'= d-\gamma
      \,.
   \end{equation}
   In fact, as $\Gamma$ does not pass through any of the points
   of $Z'$, we know that $X_3'$ has multiplicity at least three on $Z'$.

   By Lemma~\ref{lem:jump must be 1 at least}, we obtain
   from \eqnref{eqn:alpha-Z'-3} the inequality
   \begin{equation}\label{eqn:alpha-Z'-3-less}
      \alpha(\cali_{Z'})\le d-\gamma-2
   \end{equation}
   and this shows with \eqnref{eqn:alpha-Z'} that equality holds
   in \eqnref{eqn:alpha-Z'-3-less}.
   From \eqnref{eqn:alpha-Z'-3} we see then, again with
   Lemma~\ref{lem:jump must be 1 at least},
   that equality holds in
   \eqnref{eqn:alpha-Z'-3} as well.
   We have thus established that the assumptions of the theorem are
   satisfied for the set $Z'$.
   By induction we conclude therefore that $Z'$ is a star
   configuration and that $X_3'$ decomposes into planes
   or $Z'$ is contained in a hyperplane and then the support
   of $X_3'$ is that hyperplane.
   As $\Gamma$ was chosen of minimal degree,
   it must therefore be a plane as well,
   and hence $X_3$ decomposes entirely into planes.
   We can then run the above induction argument
   for \emph{any} plane component $\Pi$ of $X_3$ to see that the surface
   $X_3-\Pi$ yields a star configuration.
   This shows immediately that there are no triple
   intersection lines among these planes and we conclude
   by just counting points with multiplicity at least $3$
   (in fact exactly $3$) that
   $Z$ is a star configuration.
\end{proof}


\section{Further results and examples}\label{sect:further}

   Recall that the \emph{Waldschmidt constant} of a homogeneous ideal
   $I\subset\P^n$
   is the asymptotic counterpart of the initial degree,
   defined as
   $$
      \hat\alpha(I)=\lim\limits_{k\to\infty}\frac{\alpha(I^{(k)})}{k}=\inf\limits_{k\geq 1}\frac{\alpha(I^{(k)})}{k}
      \,.
   $$
   This invariant is indeed well defined,
   see \cite{BocHar10a} for this fact and some basic properties of
   $\hat\alpha$.
   The Waldschmidt constants are interesting invariants that
   were recently rediscovered and studied by Bocci and Harbourne,
   see e.g. \cite{BocHar10b}.
   While
   Harbourne introduced the notation
   $\gamma(\cali)$, we propose here the notation $\hat{\alpha}(\cali)$,
   as the Waldschmidt constant \emph{is} the asymptotic version
   of the initial degree $\alpha(\cali)$, and the notation is then
   consistent with \cite{dFKL07}.

   We state now a corollary of Theorem \ref{thm:main} dealing with the case when
   there is just one more $\alpha$-jump by $1$.

\begin{corollary}
   Let $Z$ be a finite set of points in projective three-space~$\P^3$ and let
   $\cali$ be the radical ideal defining $Z$. If
   \begin{equation}\label{eq:more jumps by 1}
      d:=\alpha(\cali^{(4)})=\alpha(\cali)+3
      \,,
   \end{equation}
   then
   $\alpha(\cali)=1$, i.e., $Z$ is contained in a single plane in $\P^3$.
\end{corollary}

\begin{proof}
   Using Theorem \ref{thm:main} we need to exclude the possibility that $Z$
   forms a star configuration. To this end we apply
   in the case $n=3$
   the inequality
   \begin{equation}\label{eq:chudnovsky a la demailly}
      \hat{\alpha}(\cali)\geq \frac{\alpha(\cali)+n-1}{n}
      \,,
   \end{equation}
   which was proved by Demailly \cite[Proposition 6]{Dem82}. The assumptions
   of his result are satisfied by our Theorem \ref{thm:main}.
   Combining \eqnref{eq:chudnovsky a la demailly} with the fact that
   $\hat{\alpha}(\cali)\leq\frac{\alpha(\cali^{(n+1)})}{n+1}$ yields in our situation
   $$\frac{\alpha(\cali)+n-1}{n}\leq \frac{\alpha(\cali)+n}{n+1}\,,$$
   which gives immediately $\alpha(\cali)=1$.
\end{proof}

\begin{remark}[Waldschmidt constant of star configuration]
   Note that there is equality in \eqnref{eq:chudnovsky a la demailly}
   for star configurations of points by \cite[Proof of Theorem 2.4.3]{BocHar10a}.
\end{remark}

\begin{remark}
   If Conjecture~\ref{conj:main} holds for any $n$, then the
   proof of the above corollary shows that
   if $\alpha(\cali^{(n+1)})=\alpha(\cali)+n$, then
   $\alpha(\cali)=1$.
\end{remark}

   We next provide examples showing that
   a single $\alpha$-jump by 1
   is not sufficient in order to get the conclusion of
   Theorem~\ref{thm:main}.

\begin{example}[Kummer surface]\label{ex:kummer surface}
    In this example we show that in general the assumption
    \begin{equation}\label{eq:jump kummer}
       \alpha(\cali^{(2)})=\alpha(\cali)+1
    \end{equation}
    for an ideal $\cali$ of a set of points $Z$ in $\P^3$
    is not sufficient in order to conclude that
    the points in $Z$ are coplanar or form a star configuration.

    To this end, let $X\subset\P^3$ be the classical Kummer surface
    associated with an irreducible principally polarized abelian surface,
    and let $Z$
    be the set of the $16$ double points on $X$. It is well
    known (see e.g.~\cite[Sect.~10.2]{CAV}) that these $16$
    points form a $16_6$ configuration, i.e., there are $16$
    planes $\Pi_i$ in $\P^3$ such that each plane $\Pi_i$
    contains exactly $6$ double points of $X$ (and exactly
    $6$ planes pass through every point in $Z$). We claim that
    $$\alpha(\cali)=3\quad\mbox{and}\quad\alpha(\cali^{(2)})=4\,,$$
    where $\cali$ is the radical ideal of $Z$.

    Granting this for a moment, we see immediately that
    the points in $Z$ are neither coplanar nor form
    a star configuration, whereas the assumption in
    \eqnref{eq:jump kummer} is satisfied.

    Turning to the proof, assume that there exists a
    surface $S$ defined by an element of degree $3$ in $\cali^{(3)}$.
    Let $\Pi$ be one of the $16$ planes $\Pi_i$. Then
    $$1\cdot 3\cdot 4=\Pi\cdot S\cdot X\geq 6\cdot 1\cdot 2\cdot 2$$
    implies that $\Pi$ is a component of $S$. As the same
    argument works for all $16$ planes, we get a contradiction.
    Hence $\alpha(\cali^{(2)})=4$.

    A similar argument excludes the
    possibility that $Z$ is contained in a quadric. We leave the
    details to the reader.
\end{example}

   The following simpler example exhibiting the same phenomenon
   has been suggested by the referee.

\begin{example}
   Let $L_1,L_2,L_3$ be mutually distinct and not coplanar lines in $\P^3$ intersecting
   in a point $P$. Let $A,B \in L_1$, $C,D \in L_2$ and $E\in L_3$ be points
   on these lines different from their intersection point $P$. Obviously
   the set $Z=\left\{A,B,C,D,E\right\}$ is not a star configuration.
   Let $I$ be the radical ideal of $Z$. Then it is elementary to check that
   $$\alpha(I)=2\;\mbox{ and }\; \alpha(I^{(2)})=3.$$
   Note that $5$ is the minimal number of non-coplanar points that can give
   $\alpha(I^{(2)})=\alpha(I)+1$ and not form a star configuration.
\end{example}

   The next example has been also suggested by the referee and replaces
   a much more complicated example of our original draft.

\begin{example}[Five general points in $\P^3$]\label{ex:five points}
    In this example we show that in general the assumption
    \begin{equation}\label{eq:jump five points}
       \alpha(\cali^{(3)})=\alpha(\cali^{(2)})+1
    \end{equation}
    for an ideal $\cali$ of a set of points $Z$ in $\P^3$
    is also not sufficient in order to conclude that
    the points in $Z$ are coplanar or form a star configuration.

    To this end let $Z=\left\{A,B,C,D,E\right\}$ consist
    of $5$ points in general linear position in $\P^3$.
    For the radical ideal $I$ of $Z$ one has then
    $$\alpha(I^{(3)})=\alpha(I^{(2)})+1.$$
    In fact, one has in this case
    $$\alpha(I)=2,\;\; \alpha(I^{(2)})=4\; \mbox{ and }\; \alpha(I^{(3)})=5.$$
\end{example}


\paragraph*{Acknowledgement.}
   We would like to thank Jean-Pierre Demailly for bringing reference \cite{Dem82}
   to our attention. Further, we thank
   Brian Harbourne for helpful discussions.
   We thank also the referee for detailed helpful remarks, comments
   and nice examples for Section \ref{sect:further}.



\bigskip
  Thomas Bauer,
  Fach\-be\-reich Ma\-the\-ma\-tik und In\-for\-ma\-tik,
  Philipps-Uni\-ver\-si\-t\"at Mar\-burg,
  Hans-Meer\-wein-Stra{\ss}e,
  D-35032~Mar\-burg, Germany

  \textit{E-mail address:} \texttt{tbauer@mathematik.uni-marburg.de}

\bigskip
  Tomasz Szemberg,
  Instytut Matematyki UP,
  Podchor\c a\.zych 2,
  PL-30-084 Krak\'ow, Poland

 \textit{E-mail address:} \texttt{tomasz.szemberg@gmail.com}


\end{document}